\documentclass[a4paper,11pt]{amsart}

\RequirePackage[english]{babel}
\usepackage{amsthm}

\newcommand{\Set}[1]{\left\{\, #1 \,\right\}}

\DeclareMathOperator{\Sym}{Sym}
\DeclareMathOperator{\Alt}{Alt}
\DeclareMathOperator{\Aff}{AGL}
\DeclareMathOperator{\GL}{GL}
\DeclareMathOperator{\AGL}{AGL}
\DeclareMathOperator{\T}{T}
\DeclareMathOperator{\GF}{\mathbf{GF}}

\renewcommand{\phi}[0]{\varphi}
\renewcommand{\theta}[0]{\vartheta}
\renewcommand{\epsilon}[0]{\varepsilon}

\newcommand{\N}{\text{$\mathbf{N}$}}
\newcommand{\MM}{\mathcal{M}}

\theoremstyle{definition}

\newtheorem{dummy}{Dummy}
\numberwithin{dummy}{section}
\numberwithin{equation}{section}

\newtheorem{theorem}[dummy]{Theorem}
\newtheorem{lemma}[dummy]{Lemma}

\newtheorem{cor}[dummy]{Corollary}
\newtheorem{conjecture}[dummy]{Conjecture}
\newtheorem{proposition}[dummy]{Proposition}

\theoremstyle{definition}

\newtheorem{defin}[dummy]{Definition}

\theoremstyle{remark}

\newtheorem{fact}{Fact}

\newtheorem{remark}[dummy]{Remark}

\newcommand{\FF}{{\mathbb F}}

\begin{document}

\date{June 2008 --- Version 6.06%
}

\title[Ciphers and Imprimitivity]
{On some block ciphers and imprimitive groups}

\author{A. Caranti}

\address[A. Caranti]{Dipartimento di Matematica\\
  Universit\`a degli Studi di Trento\\
  via Sommarive 14\\
  I-38100 Povo (Trento)\\
  Italy} 

\email{caranti@science.unitn.it} 

\urladdr{http://www-math.science.unitn.it/\~{ }caranti/}

\author{F. Dalla Volta}

\address[F.~Dalla Volta]{Dipartimento di Matematica e Applicazioni\\
  Edificio U5\\
  Universit\`a degli Studi di Milano--Bicocca\\
  Via R.~Cozzi, 53\\
  I-20126 Milano\\
  Italy}

\email{francesca.dallavolta@unimib.it}

\urladdr{http://www.matapp.unimib.it/\~{ }dallavolta/}

\author{M. Sala}

\address[M. ~Sala]{Dipartimento di Matematica\\
  Universit\`a degli Studi di Trento}

\email{msala@bcri.ucc.ie}

\begin{abstract}
The group generated by the round functions of a block ciphers is a widely
investigated problem. We identify a large class of block ciphers
for which such group is easily guaranteed to be primitive. Our class
includes the AES and the SERPENT.
\end{abstract}

\keywords{AES, key-alternating block ciphers, primitive groups}

\thanks{Caranti   and  Dalla  Volta   are  members   of  INdAM-GNSAGA,
 Italy. Caranti  has been partially  supported by MIUR-Italy  via PRIN
 2001012275 ``Graded  Lie algebras  and pro-p-groups of  finite width,
 loop  algebras, and  derivations''.  Dalla  Volta has  been partially
 supported    by   MIUR-Italy    via   PRIN    ``Group    theory   and
 applications''.    Sala    has    been   partially    supported    by
 STMicroelectronics contract ``Complexity  issues in algebraic Coding
 Theory and Cryptography''}

\maketitle

\section{Introduction}

Most block ciphers are iterated block ciphers, i.e. they are obtained
by the composition of several ``rounds'' (or ``round functions'').
A round is a key-dependent permutation of the message/cipher space.
To achieve efficiency, all rounds share a similar structure.

For a given cipher, it is an interesting problem to determine the permutation
group generated by its round functions (with the key varying in the key space),
since this group might reveal weaknesses of the cipher.
However, these results usually require an ad-hoc proof 
(with a notable recent exception \cite{SparrWerns}).

In this paper we consider a class of block ciphers, large enough to contain
some well-known ciphers (like the AES and the SERPENT), which is such that
the primitivity of the related group can be easily established 
by only checking some properties of its S-Boxes.
Our results may be useful to cipher designers wanting to avert
group imprimitivity, since in our context they would do it easily.



\section{Preliminaries}
\label{sec:prelim}

\subsection{Group theory and finite field theory}

Let $G$ be a finite group acting transitively on a set $V$ and $H\leq G$ 
a subgroup. 
We write the action of an element $g \in G$ on an element $\alpha \in V$
as $\alpha g$. 
Also, $\alpha G = \Set{ \alpha g : g \in G}$ is the orbit of $\alpha$
and $G_{\alpha} = \Set {g \in G : \alpha g =\alpha }$ is its stabilizer.
A  \emph{partition} $\mathcal{B}$ of $V$ is \emph{$G$-invariant} 
if for any $B  \in \mathcal{B}$ and $g \in
G$,  one  has  $B   g  \in  \mathcal{B}$. 
Partition $\mathcal{B}$ is \emph{trivial} if $\mathcal{B}  = \Set{ V
}$ or $\mathcal{B} = \Set{ \Set{\alpha} : \alpha \in V}$.
If $\mathcal{B}$ is non-trivial then it is a \emph{block system} 
for the action of $G$ on $V$ (and any $B\in {\mathcal B}$ is a {\em{block}}).
 If such a block system exists,
then we say that $G$ is \emph{imprimitive} in its action on $V$
(equivalently, \emph{$G$ acts imprimitively} on $V$).
If $G$ is not imprimitive (and it is transitive), then we
say that it is \emph{primitive}.
Since $G$ acts transitively on $V$, we have then $\mathcal{B}
= \Set{ B g : g \in G }$.
\begin{lemma}[\cite{Cam}, Theorem~1.7]
\label{lemma}
Let $G$ be a finite group, acting transitively on a set $V$.
Let $\alpha \in V$.
Then the blocks $B$ containing $\alpha$ are in one-to-one correspondence
with the subgroups $H$ such that with $G_{\alpha} < H < G$. The correspondence
is given by $B = \alpha H$.

In particular, $G$ is primitive if and only if $G_{\alpha}$ is a
maximal subgroup of $G$.
\end{lemma}



We denote by $\Sym(V)$ and $\Alt(V)$, respectively, the symmetric and 
alternating group on $V$.
When $V$ is a vector space over a finite field $\FF_q$ with $q$ elements,
we also denote by $\T(V)$ the translation group
$\T(V)=\Set{\sigma_{v} : v \in V}$, where 
  $\quad \sigma_{v} : V \to V,\quad w\mapsto w + v \,$.
It is well-known that $\T(V)$ is a transitive
subgroup of $\Sym(V)$, which is imprimitive except for the trivial case
$V=\FF_p$, with $p$ a prime.
Any block system ${\mathcal B}$ of $\T(V)$ is the set of
translates of a proper vector subspace $W$ of $V$, that is, 
${\mathcal{B}}=\{W+v \mid v\in V\}$.
We denote by $\Aff(V)$ the group of all affine permutations
of $V$, which is a primitive maximal subgroup of $\Sym(V)$, and
by $\GL(V)$ the group of all linear permutations of $V$,
which is a normal subgroup of $\Aff(V)$.

We will need the following result from finite field theory.
\begin{theorem}[\cite{Matt},\cite{Others}]
\label{theorem:Sandro}
  Let $\FF$ be a field of characteristic two. Suppose  $U \ne 0$ is an
  additive subgroup of $\FF$ which contains the inverses of each of its
  nonzero elements. Then $U$ is a subfield of $\FF$.
\end{theorem}
\ 

\subsection{Vectorial Boolean functions}

Let $m\geq 1$ be a natural number. Let $A=(\FF_2)^m$ and $A^*=A\setminus\{0\}$.
Any function $F:A\to A$ is a \emph{vectorial Boolean function} (vBf).\\
For any function $F:A\to A$ and any elements $a,b\in A $, $a\not=0$,
we denote
$$
\delta_F(a,b)  =  |\{x\in A: F(x+a)+F(x)=b\}| \;.
$$
Let $\delta\in \N$.
Function $F$ is called a \emph{differentially $\delta$-uniform} function 
(\cite{Ny}) if
$$
  \forall a\in A^*,\forall b\in A,\qquad
  \delta_F(a,b) \le \delta \;.
$$
The smallest such $\delta$ is called the \emph{differential uniformity}
of $F$.
Note that $\delta\ge 2$ for any vBf. 
Differentially 2-uniform mappings are called \emph{almost perfect nonlinear},
or APN for short.
If we denote by $\hat F_a$ the vBf which maps  
$x \mapsto F(x  + a)   + F(x)$, then $F$ is differential 
$\delta$-uniform if and
only if $|(\hat F_a)^{-1}(b)|\leq \delta$ (for any $a$ and $b$).
From now on, we shorten ``differential uniformity'' to ``uniformity''.

Vectorial Boolean functions used as S-boxes in block ciphers must have 
low uniformity  to prevent differential cryptanalysis 
(see \cite{B-Sh,Ny}).  In this sense, APN  functions are optimal. 
However, numerous experiments suggest the following conjecture
\begin{conjecture}[Dobbertin]
If $m$ is even, no APN function is a permutation.
\end{conjecture}
If this conjecture is true, then APN functions cannot be used as S-Boxes,
since implementation issues require an even $m$.

Any vBf can also be regarded as a polynomial in $\FF_{2^m}[x]$
(with degree at most $2^m-1$).
When $m$ is even,  the {\it patched inverse} function $x^{2^m-2}$ is a
4-uniform permutation  (\cite{Ny}) and was chosen 
as the basic S-box, with $m=8$, in the Advanced Encryption Standard (AES)
(\cite{AES}).
\\

\subsection{Previous results on the group generated by the round functions}

Let ${\mathcal C}$ be any block cipher such that the plain-text space $\MM$
coincides with the cipher space. Let ${\mathcal K}$ be the key space.
Any key $k\in {\mathcal K}$ induces a permutation $\tau_k$ on $\MM$.
Since $\MM$ is usually $V=(\FF_2)^n$ for some $n\in \N$, we can consider
$\tau_k\in \Sym(V)$.
We denote by $\Gamma=\Gamma({\mathcal C})$ the subgroup of $\Sym(V)$ generated
by all the $\tau_k$'s.
In literature the following properties of $\Gamma$ are considered undesirable,
since they could lead to weaknesses of ${\mathcal C}$: small cardinality,
imprimitivity and  intransitivity. For a detailed discussion of their
consequences, see \cite{SparrWerns}. We would add that $\Gamma$ should not be
a subgroup of  $\Aff(V)$, otherwise it is obvious how to break
the cipher.
If $\Gamma$ turns out to be $\Alt(V)$ or $\Sym(V)$, these properties
are automatically avoided.
Note also that primitivity alone guarantees a non-negligible group size, 
but it could still be that $\Gamma$ would be weak 
(as for example if $\Gamma \leq \Aff(V)$).

Unfortunately, the knowledge of $\Gamma({\mathcal C})$ is out of reach
for the most important ciphers (such as the AES, the SERPENT, the DES,
the IDEA). However, researchers have been able to compute another related
group. Suppose that ${\mathcal C}$ is the composition of $l$ rounds.
\begin{remark}
Note that the division into rounds is not mathematically
well-defined, but it is provided in the document describing the cipher,
so this division is debatable and a cryptanalyst is allowed
to modify it, if it is convenient.
\end{remark}
Then any key $k$ would induce $l$ permutations, 
$\tau_{k,1},\ldots,\tau_{k,l}$, whose composition is $\tau_k$.
For any round $h$, we can consider $\Gamma_h({\mathcal C})$ as the 
subgroup of $\Sym(V)$ generated by the $\tau_{k,h}$'s 
(with $k$ varying in $V$).
We can thus define the group $\Gamma_{\infty}=\Gamma_{\infty}({\mathcal C})$
as the subgroup of $\Sym(V)$ generated by all the $\Gamma_h$'s.
We note the following elementary fact.
\begin{fact}
\label{gr}
$
  \Gamma \,\leq\, \Gamma_{\infty} \,. 
$
\end{fact}
Group $\Gamma_{\infty}$ is traditionally called the \emph{group generated
by the round functions}. Note that independent sub-keys are implicitly 
assumed.
We collect in the following proposition some previous
results on $\Gamma_{\infty}$.
\begin{proposition}
\label{wer}
\
\begin{itemize}
\item $\Gamma_{\infty}({\mathrm{AES}})=\Alt(V)$ \cite{WAES},
\item $\Gamma_{\infty}({\mathrm{SERPENT}})=\Alt(V)$ \cite{Wserpent},
\item $\Gamma_{\infty}({\mathrm{DES}})=\Alt(V)$ \cite{WDES}.
\end{itemize}
\end{proposition}
The proof of any of the results in Proposition \ref{wer} requires
an ad-hoc proof.
Recently, a generalization of some of these results have
been proposed \cite{SparrWerns}.

\section{A class of block ciphers}

Several definitions have  been proposed for iterated block ciphers 
(see e.g. \emph{key-alternating block  cipher} in \cite{AES}, or
 Rjindael-like ciphers in \cite{SparrWerns}).
We would like to define a class, large enough to include most common ciphers, 
yet restricted enough to have simple criteria guaranteeing the primitivity
of $\Gamma_{\infty}$.

Let  ${\mathcal C}$ be a block cipher with $V=(\FF_2)^n$ and
$n=m s$, $s\geq 2$.
Space $V$ is a direct sum
\begin{equation*}
  V = V_{1} \oplus \dots \oplus V_{s},
\end{equation*}
where each $V_{i}$ has the same dimension $m$ (over $\FF_2$). 
For any $v\in V$, we will write $v = v_{1} \oplus \dots \oplus v_{s}$, 
where $v_{i} \in V_{i}$. 
Also, we consider the projections $\pi_{i} : V \to V_{i}$ mapping
$v \mapsto v_{i}$.
Any $\gamma\in \Sym(V)$ that acts as
\begin{equation*}
  v \gamma = v_{1} \gamma_{1} \oplus \dots \oplus v_{s} \gamma_{s},
\end{equation*}
for some $\gamma_{i}\in \Sym(V_i)$, is a \emph{bricklayer transformation}
and any $\gamma_i$ is a {\em brick}.
When used in symmetric cryptography, maps $\gamma_i$'s are traditionally 
called {\em S-boxes} and map $\gamma$
is called a ``parallel S-box''.

A linear (or affine) map $\lambda:V\rightarrow V$ is traditionally
called a ``mixing layer'',
when used in composition with parallel maps.
\\

In the following definitions we are not following established
notation.\\
We call any linear map $\lambda\in \GL(V)$  a \emph{proper mixing layer} if
no  sum of  some  of the  $V_{i}$  (except $\Set{0}$  and $V$)  is
    invariant under $\lambda$. A similar definition can be given
when $\lambda\in \AGL(V)$.
\\

We define our class.
\begin{defin} 
We say that ${\mathcal C}$ is 
\emph{translation based} (tb) if it is the composition of some rounds, 
such that any is of the form $\tau_{k,h}=\gamma_h\lambda_h\sigma_k$,
with $k\in V$ 
($\gamma_h$ and $\lambda_h$ do not depend on $k$, but they might depend on
 the round), where $\gamma_h$ is a bricklayer transformation
and $\lambda_h$ is a linear map
(but $\lambda_h$ is a proper mixing layer for at least one round).
\end{defin}
A round when the mixing layer is proper is called a {\em proper} round.
\begin{remark}
A round consisting of only a translation is still acceptable, by taking
$\gamma_h=\lambda_h=1_V$ (the identity map on $V$), although
obviously it is not proper.
\end{remark} 
The previous definition is similar to \emph{key-alternating block  cipher}
(see Section~2.4.2 of~\cite{AES}), 
although the latter is too general for our goals.
\\

From now on, we assume ${\mathcal C}$ is a tb cipher and that $0\gamma=0$
(this can always be assumed).
From the knowledge of block systems of $T(V)$, we immediately obtain
the following.
\begin{fact}
  \label{fact:impr}
  Let  $G=\Gamma_h({\mathcal C})$ for any round $h$.
  Then $T=T(V)\subset G$.
  Therefore, if $G$ acts imprimitively on ${\mathcal M}=V$, 
  the blocks of imprimitivity  are the translates of a linear subspace.
\end{fact}
\begin{proof}
We show $T\subset G$.
For any $k\in V$, we have $\gamma_h\lambda_h\sigma_k\in G$.
By considering the zero key, we have also 
 $\gamma_h\lambda_h\sigma_0=\gamma_h\lambda_h\in G$.
 Therefore,  
$(\gamma_h\lambda_h)^{-1} \gamma_h\lambda_h\sigma_k=\sigma_k\in G$.
\end{proof}
\begin{cor}
  \label{fact:UinU}
  Let  $G=\Gamma_h({\mathcal C})$ for any round $h$.
  Then $G$ acts imprimitively if and only if
  there is a subspace $U < V$ ($U \ne \{0\}, V$) such that for any
  $v \in V$ and $u \in U$, we have
  \begin{equation}
    (v + u) \gamma_h\lambda_h + v \gamma_h\lambda_h \in U.
  \end{equation}
\end{cor}
\begin{proof}
$G$ is imprimitive if and only if there is a block system
of type $\{v+U\}$, for some subspace $U$, $U \ne \{0\}, V$.  \\
It is enough to consider a zero round key, so that
$$
(v+U)\gamma_h\lambda_h\sigma_0=v\gamma_h\lambda_h\sigma_0 +U \implies
(v+U)\gamma_h\lambda_h=v\gamma_h\lambda_h +U \;.
$$
\end{proof}

\section{Main results}

We define for a vBf $f$ two new notions of non-linearity.
The first is weaker than $\delta$-uniformity.
\begin{defin}
For any $m\geq 2$ and $\delta\geq 2$, let $A=(\FF_2)^m$ and $f\in \Sym(A)$.
We say that $f$ is 
{\bf weakly $\delta$-uniform} if for any $u\in A$, $u\not=0$,
the size of image of $\hat f_u$ is at least
$$
  |\mathrm{Im}(\hat f_u)| \geq \frac{2^m}{\delta+2}+1 \,.
$$ 
\end{defin}
It is trivial to prove that a $\delta$-uniform map is indeed
weakly $\delta$-uniform.
\begin{proof}
Let $B=\mathrm{Im}(\hat f_u)$.
If $f$ is $\delta$-uniform, then $|(\hat f_u)^{-1}(b)|\leq \delta$,
for any $b\in B$.\\
From $A=\sqcup_{b\in B} (\hat f_u)^{-1}(b) $, we have
$$
    A=\sqcup_{b\in B} (\hat f_u)^{-1}(b) \implies 2^m=|A|=
    \sum_{b\in B} |(\hat f_u)^{-1}(b)| \leq \delta |B|  \,
$$
which means
$$
  |B| \geq \frac{2^m}{\delta} > \frac{2^m}{\delta+2} \,. 
$$
\end{proof}
\begin{remark} \label{rem-weak}
If a function $f$ is weakly $\delta$-uniform, 
with $2^r\geq \delta$ and the image 
$\mathrm{Im}(\hat f_u)$  is contained in a subspace $W$, then
the dimension of $W$ is at least $m-r$.
This is the property of $f$ which will be needed in the proof
of Theorem \ref{theorem:main}.
Interestingly, if $f$ is $\delta$-uniform  (as in Subsection 2.2),
then the dimension of $W$ which can be guaranteed is exactly the same
(and not any bigger).
\end{remark}

Our second notion focuses on the image of vector spaces.
\begin{defin}
Let $A=(\FF_2)^m$.
We say that $f$ is $l$-{\bf anti-invariant} if for any subspace
$U\,\leq\,A$ such that $f(U)=U$ we have
$\dim(U)<m-l$ or $U=A$.

We say that $f$ is {\bf strongly} $l$-anti-invariant, if
for any two subspaces $U,W \,\leq\, A$, such that $f(U)=W$,
we have  $\dim(U)=\dim(W)<m-l$ or $U=W=A$.
\end{defin}

In other words, $l$-anti-invariant means that
the largest subspace invariant under $f$ has 
codimension greater than $l$ (except for $A$ itself),
while strongly $l$ anti-invariant means that the largest subspace 
sent by $f$ into another subspace
has codimension greater than $l$ (except for $A$ itself).
\\

We are ready for our main result (recall that $0 \gamma = 0$).

\begin{theorem}
  \label{theorem:main}
  Let $C$ be a tb cipher, with $\lambda_h$ a proper mixing layer,
  and $G=\Gamma_h(C)$.
  Let $1 \le r < m/2$.
  If any brick of $\gamma_h$ is weakly 
   $2^r$-uniform and strongly $r$-anti-invariant,
  then $G$ is primitive and hence $\Gamma_{\infty}(C)$
  is primitive.
\end{theorem}
\begin{proof}
We drop the $h$-underscript in this proof and
we suppose, by way of contradiction, that $G$ is imprimitive.

Let $U$ be any proper subspace of $V$ s.t. $\{v+U\}_{v\in V}$ 
form a block system
 for $G$. Since $U$ is a block and $\gamma\lambda\in G$,
we have $U\gamma\lambda=U+v$ for some $v\in V$.
But $0\gamma\lambda=0\in U+v$, so $v=0$ and 
\begin{equation}
\label{uv}
U\gamma\lambda=U \;.
\end{equation}

 Let $I$ be the set of all $i$ s.t.
$\pi_i(U)\not=0$. Clearly, $I\not=\emptyset$.
Then:
\begin{itemize}
\item either $U\cap V_i=V_i$ for all $i\in I$,
\item or there is $\iota\in I$ s.t. $U\cap V_\iota\not= V_\iota$.
\end{itemize}
In the first case, $U=\oplus_I V_i$, which means $U\gamma=U$.
But (\ref{uv}) implies $U\lambda=U$, which is impossible since 
$\lambda$ is a proper mixing layer.

In the second case, we denote $W=U\gamma$ (equal
to $U\lambda^{-1}$ by (\ref{uv})) and we note
that 
\begin{equation}\label{ovvio}
(U\cap V_\iota)\gamma'=W\cap V_\iota \,,
\end{equation}
where $\gamma'=\gamma_\iota$ is the brick of $\gamma$ in $V_\iota$.
By Corollary \ref{fact:UinU}, we have that
$B=\mathrm{Im}(\hat \gamma'_u)\subset W\cap V_\iota$ for any 
$u\in U\cap V_\iota$.
But $\gamma'$ is weakly $2^r$-uniform, so (Remark \ref{rem-weak})
$\dim(W\cap V_\iota)=\dim(U\cap V_\iota)\geq m-r$.
By (\ref{ovvio}), this is impossible, since $\gamma'$ is
strongly $r$-anti-invariant.

\end{proof}

To apply our theorem to the AES, we first need a simple lemma.

\begin{lemma}\label{2r}
Let $f$ be a vBf.
If $f^2=1$ and $f$ is $2r$-anti-invariant with $1\leq r<m/2$, 
then $f$ is strongly $r$-anti-invariant.
\begin{proof}
Let $U,W$ be subspaces of codimension $l$ such that $U f=W$.
Let us consider $Z=U\cap U f$. By standard linear algebra,
$\dim(Z)\geq n-2l$.  Since $Z f=Z$ and $f$ is $2r$-anti-invariant,
$l$ must be $l>r$, and so $U$ and $W$ have codimension strictly
bigger than $r$. 
\end{proof}
\end{lemma}
\ \\

The first interesting consequence of our theorem is the following.
\begin{cor}
  \label{cor:AES_is_OK}
  Any typical round $h$ of the AES 
  satisfies the hypotheses of Theorem~\ref{theorem:main}.
  As a consequence, 
  both $\Gamma_h(AES)$ and $\Gamma_\infty(AES)$ are primitive.

\end{cor}

\begin{proof}
  We first show that the mixing layer $\lambda=\lambda_h$ of a typical
  round of the AES is proper.
  Suppose $U \ne \Set{0}$ is a  subspace of $V$ which
  is invariant  under $\lambda$. Suppose, without  loss of generality,
  that   $U   \supseteq   V_{1}$.   Because   of   \texttt{MixColumns}
  \cite[3.4.3]{AES}, $U$ contains the whole first column of the state.
  Now   the  action   of   \texttt{ShiftRows}  \cite[3.4.2]{AES}   and
  \texttt{MixColumns} on the first column shows that $U$ contains four
  whole columns, and  considering  (if the state has more than
  four  columns)  once  more  the  action  of  \texttt{ShiftRows}  and
  \texttt{MixColumns}, one sees immediately that $U = V$.

  The S-box $\gamma'$ is well-known to satisfy (for any $u\not=0$)
  $\mathrm{Im}(\hat \gamma'_u)=2^7-1\geq 2^6+1$ and so it is
  weakly $2$-uniform.
  
  To apply the theorem we need only to show that $\gamma'$ is
  strongly $1$-anti-invariant. Since $(\gamma')^2=1$, we want to apply
  Lemma \ref{2r} with $r=1$. Indeed, $\gamma'$ is well-known to
  be $3$-anti-invariant, since the  only  nonzero  
  subspaces  of  $\GF(2^{8})$ which are invariant under inversion 
  are the subfields (Theorem \ref{theorem:Sandro}), 
  and so the largest proper one is $\GF(2^{4})$,  of codimension $4 > 3$.

\end{proof}
\ \\

The second interesting consequence is the following.
\begin{cor}
  \label{cor:SERPENT_is_OK}
  Any typical round $h$ of the SERPENT 
  satisfies the hypotheses of Theorem~\ref{theorem:main}.
  As a consequence, 
  both $\Gamma_h(SERPENT)$ and $\Gamma_\infty(SERPENT)$ are primitive.

\end{cor}
\begin{proof}
The conditions of Theorem~\ref{theorem:main} are satisfied with
$r=1$, as can be seen by a direct computer check on all Serpent S-boxes
and on its mixing layer (\cite{ilaria}).
\end{proof}


\subsection*{Acknowledgements}

We  are grateful to  P.~Fitzpatrick, L.~Knudsen and  C.~Traverso 
for  their useful comments. 
Part of this work was presented at the Workshop on
Coding and Cryptography, which was held in UCC, Cork (2005).

\providecommand{\bysame}{\leavevmode\hbox to3em{\hrulefill}\thinspace}
\providecommand{\MR}{\relax\ifhmode\unskip\space\fi MR }
\providecommand{\MRhref}[2]{%
  \href{http://www.ams.org/mathscinet-getitem?mr=#1}{#2}
}
\providecommand{\href}[2]{#2}

\end{document}